\documentclass{amsart}

\usepackage{amssymb}
\usepackage{graphicx}
\usepackage{latexsym}
\usepackage{epsfig}
\usepackage{psfrag}
\def\scfig #1 #2 {\resizebox{#2}{!}{\includegraphics{#1}}}

\def\qed{\hfill$\Box$}

\newtheorem{theorem}{Theorem}
\newtheorem{lemma}[theorem]{Lemma}
\newtheorem{proposition}[theorem]{Proposition}
\newtheorem{corollary}[theorem]{Corollary}

\theoremstyle{definition}
\newtheorem{definition}[theorem]{Definition}

\theoremstyle{remark}
\newtheorem{remark}[theorem]{Remark}
\newtheorem{conjecture}[theorem]{Conjecture}

\begin{document}

\title[Guionnet-Jones-Shlyakhtenko subfactors associated to Kac algebras]{Guionnet-Jones-Shlyakhtenko subfactors associated to
finite-dimensional Kac algebras}

\author{Vijay Kodiyalam}
\address{The Institute of Mathematical Sciences, Taramani, Chennai, India 
600113}
\email{vijay@imsc.res.in}

\author{V. S. Sunder}
\address{The Institute of Mathematical Sciences, Taramani, Chennai, India 
600113}
\email{sunder@imsc.res.in}





\begin{abstract}
We analyse the Guionnet-Jones-Shlyakhtenko construction for the
planar algebra associated to a finite-dimensional Kac algebra and
identify the factors that arise as finite interpolated free group
factors.
\end{abstract}

\maketitle



The main theorem of \cite{GnnJnsShl2008} constructs an extremal finite index $II_1$ subfactor $N = M_0 \subseteq M_1 = M$ from a subfactor planar
algebra $P$ with the property that the planar algebra of $N \subseteq M$ is isomorphic to $P$. We show in this paper that if $P = P(H)$ - 
the (subfactor) planar algebra associated with an $n$-dimensional
Kac algebra $H$ (with $n > 1$) - then, for the associated subfactor $N \subseteq M$, 
there are isomorphisms
$M \cong LF({2\sqrt{n}-1})$
and $N \cong LF({2n\sqrt{n} - 2n +1})$,
where $LF(r)$ for $r > 1$ is the interpolated free group factor of
\cite{Dyk1994} and \cite{Rdl1994}.

The first three sections of this paper are devoted to recalling
 various results we need.
In \S 1 we summarise the  Guionnet-Jones-Shlyakhtenko (henceforth GJS)
construction.
We discuss, in \S 2, a presentation of the planar algebra associated to a finite-dimensional Kac algebra in terms of  generators and relations.
The goal of \S 3 is to collect together results that we use from free probability theory.
The longer sections, \S 4 and \S 5 are devoted to analysing the structure of the
factors $M_1$ and $M_2$ respectively.
The final \S 6 proves our main result.


\section{Guionnet-Jones-Shlyakhtenko
subfactors}
We begin with a quick review of the  GJS construction (see also 
\cite{JnsShlWlk2008} and \cite{KdySnd2008}). All tangles used in the definitions
are illustrated in Figure \ref{strmaps}. 

Suppose that $P$ is a subfactor planar algebra (see
\cite{Jns1999} or \cite{KdySnd2004} for detailed definitions) of modulus $\delta>1$. Construct a
tower of graded $*$-algebras $Gr_k(P)$ for $k \geq 0$ as follows. 
Set $Gr_k(P) = \oplus_{n=k}^\infty P_n$ and define multiplication
on $Gr_k(P)$ by requiring that if $a \in P_m \subseteq Gr_k(P)$ and
$b \in P_n \subseteq Gr_k(P)$, then $a \bullet b \in P_{m+n-k} \subseteq
Gr_k(P)$ is given by $a \bullet b = Z_M(a,b)$.
\begin{figure}[!htb]
\psfrag{M}{$M=M(k)_{m,n}^{m+n-k}$}
\psfrag{D}{$D=D(k)^n_n$}
\psfrag{I}{$I=I(k-1)_{n-1}^n$}
\psfrag{T}{$T=T(k)_{m,m-k}^0$}
\psfrag{a}{$1$}
\psfrag{b}{$2$}
\psfrag{a^*}{}
\psfrag{2m-2k}{\tiny $2m-2k$}
\psfrag{2n-2k}{\tiny $2n-2k$}
\psfrag{2(m-k)}{\tiny $2n-2k$}
\psfrag{2n-2k-2}{\tiny $2n-2k$}
\psfrag{k-1}{\tiny $k\!-\!1$}
\psfrag{2n-k}{\tiny $2n-$}
\psfrag{-k-1}{\tiny $k\!-\!1$}
\psfrag{m-k}{\tiny $m\!-\!k$}
\psfrag{T_{m-k}}{$2$}

\psfrag{k}{\tiny $k$}
\psfrag{2k}{\tiny $2k$}
\includegraphics[height=6cm]{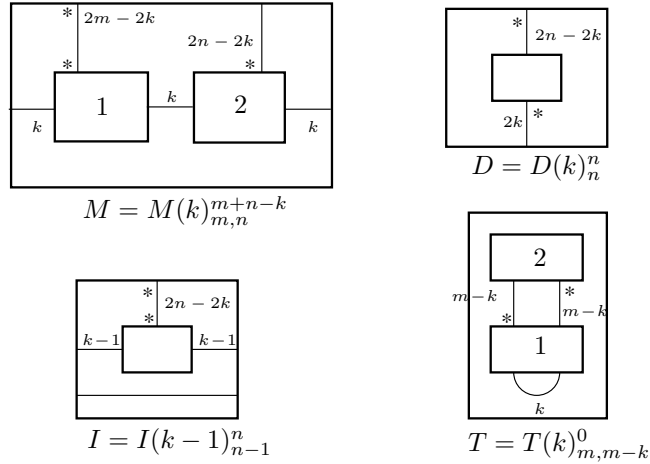}
\caption{Tangles defining structure maps of $Gr_k(P)$.}
\label{strmaps}
\end{figure}

In Figure \ref{strmaps} and other figures in this paper, we use the convention
introduced in \cite{KdySnd2008} of decorating strands in a
tangle with non-negative integers to represent cablings of that
strand. The notation for tangles such as $M=M(k)_{m,n}^{m+n-k}$ in
Figure \ref{strmaps} indicates that it is affiliated to $Gr_k(P)$, takes inputs from
$P_m$ and $P_n$ and has output in $P_{m+n-k}$.

The $*$-structure on $Gr_k(P)$ (denoted by $\dagger$ to distinguish
it from the $*$-structure of the planar algebra $P$) is defined by letting $a^\dagger \in P_n \subseteq Gr_k(P)$ be given
by $a^\dagger = Z_D(a^*)$ 
for $a \in P_n \subseteq Gr_k(P)$.
The inclusion map $Gr_{k-1}(P) \rightarrow Gr_{k}(P)$ is defined
to be the graded map whose restriction to $P_{n-1} \subseteq
Gr_{k-1}(P)$ is given by $Z_I$. 
%

%


Motivated by free probability theory (but having an entirely planar algebraic definition) is a trace $Tr_k$ defined on $Gr_k(P)$ by
letting $Tr_k(a)$ for $a \in P_m \subseteq Gr_k(P)$  be given by
$Z_T(a \otimes T_{m-k})$
%
%
%
%
where $T_m \in P_m$ is defined to be the sum of all the Temperley-Lieb
elements of $P_m$. The nomalised family $\tau_k = \delta^{-k}Tr_k$ of
traces on $Gr_k(P)$ is then consistent with the inclusions.

\begin{theorem}[see \cite{GnnJnsShl2008}]
  For each $k \geq 0$, the trace $\tau_k$ is a faithful, positive
  trace on $Gr_k(P)$. If $M_k$ denotes the von Neumann algebra
  generated by $Gr_k(P)$ in the GNS representation afforded by
  $\tau_k$, there is a tower $M_0 \subseteq M_1 \subseteq M_2
  \subseteq \cdots $ of $II_1$-factors which is the basic construction
  tower of the extremal subfactor $M_0 \subseteq M_1$
  which has index $\delta^2$  and planar algebra isomorphic to $P$.\qed
\end{theorem}

\section{The planar algebra of a Kac algebra}\label{kac}

In this section we will review (from \cite{KdyLndSnd2003}) the main facts regarding the planar algebra $P(H)$ associated to a finite
dimensional Kac algebra $H$.

For the rest of the paper, fix a Kac algebra (= Hopf $C^*$-algebra)
$H$ of finite dimension $n > 1$. The structure maps of $H$
are denoted by $\mu,\eta,\Delta,\epsilon$ and $S$. Let $H^*$ be the
dual Kac algebra of $H$ and let
$\phi \in H^*$ and $h \in H$ denote the normalised traces
in the left regular representations of $H$ and $H^*$ respectively.
These are central projections that satisfy $ah = \epsilon(a)h$, 
$\phi\psi = \psi(1)\phi$  for all $a \in H$
and $\psi \in H^*$ and further, $\phi(h) = \frac{1}{n}$.

Let $\delta = \sqrt{n}$.
Associated to $H$ is a planar
algebra  $P = P(H)$ and defined to be the quotient of the
universal planar algebra on the labelling set $L = L_2 = H$
by the set of relations in 
Figure \ref{rels}
where, (i) we write 
the relations as identities - so the statement 
$a=b$ is interpreted as $a-b$ is a relation; (ii) $\zeta \in {\mathbb C}$  and  $a,b \in H$; and
(iii) the external boxes of all tangles appearing in the relations are left 
undrawn and it is assumed that all external $*$'s are at the top left corners.

\begin{figure}[!h]
\begin{center}
\psfrag{zab}{\huge $\zeta a + b$}
\psfrag{eq}{\huge $=$}
\psfrag{a}{\huge $a$}
\psfrag{b}{\huge $b$}
\psfrag{z}{\huge $\zeta$}
\psfrag{+}{\huge $+$}
\psfrag{del}{\huge $\delta^{-1}$}
\psfrag{1h}{\huge $1_H$}
\psfrag{h}{\huge $h$}
\psfrag{epa}{\huge $\epsilon(a)$}
\psfrag{eq}{\huge $=$}
\psfrag{delinphia}{\huge $\delta \phi(a)$}
\psfrag{a1}{\huge $a_1$}
\psfrag{a2b}{\huge $ba_2$}
\psfrag{sa}{\huge $Sa$}
\resizebox{12.0cm}{!}{\includegraphics{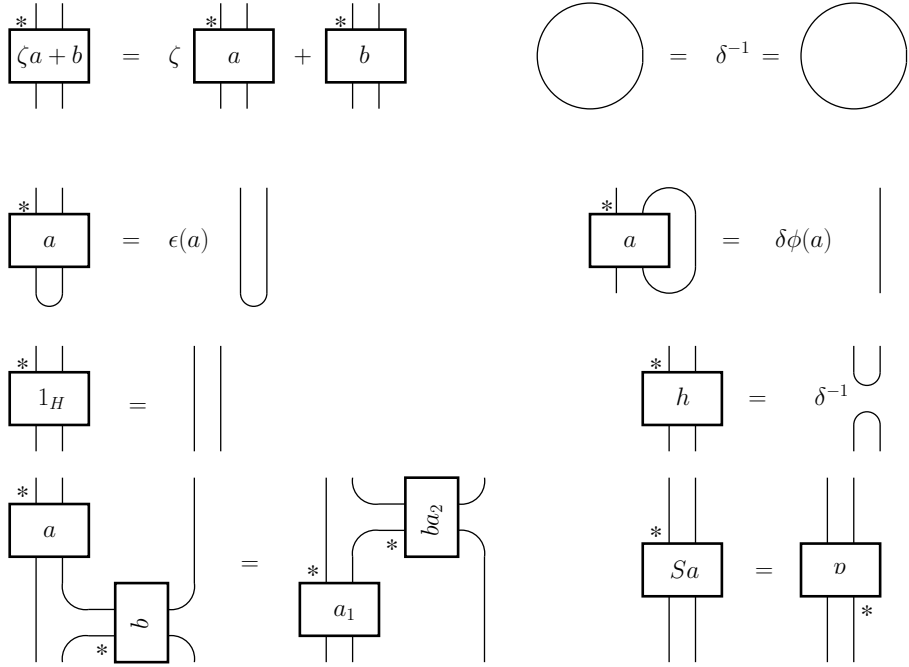}}
\end{center}
\caption{Relations in $P(H)$}
\label{rels}
\end{figure}

%

%


%

\begin{theorem}[Theorem 5.1 of \cite{KdyLndSnd2003}]\label{paha}
The planar algebra $P(H)$ is a subfactor planar algebra of modulus $\delta = \sqrt{n}$ and is the planar algebra of
the subfactor $M^H \subseteq M$ where $M$ is the hyperfinite
$II_1$-factor equipped with an outer action of $H$. There is a natural
identification of $H$ with $P_2$ under which the antipode $S$
of $H$ corresponds to the action $Z_{R}$ of the 2-rotation tangle $R = R^2_2$. \qed
\end{theorem}

We pause to remark that the convention regarding the labelling of boxes in multiplication
tangles in this paper agrees with that of \cite{GnnJnsShl2008}
and of \cite{KdySnd2008} but is opposite to that of \cite{KdyLndSnd2003} and so one of the relations here appears to be different from the corresponding one in \cite{KdyLndSnd2003}.

We will have occasion to use some other facts about $P(H)$
that depend on an explicit choice of basis for $H$.
Suppose that $\widehat{H^*}$
is a complete set of inequivalent irreducible $*$-representations of $H^*$; we will denote
a typical element of $\widehat{H^*}$ by $\gamma$ and its dimension
by $d_\gamma$. Then the set $\{\gamma_{pq} \in H : \gamma \in \widehat{H^*}, 1 \leq p,q \leq d_\gamma \}$ is a linear basis for $H$.

\begin{proposition}\label{xprop}
\begin{itemize}
\item[(1)]  Let $\gamma \in \widehat{H^*}$. Then, $\gamma_{pq}^* = S \gamma_{qp}$.
\item[(2)] The set $\{\widetilde{\gamma_{pq}} = \sqrt{d_\gamma} \gamma_{pq} : \gamma \in \widehat{H^*}, 1 \leq p,q \leq d_\gamma\}$ is an orthonormal basis of $H$
for the inner product defined by $\phi$.
\item[(3)] Let $X = X^n_{2,2,\cdots,2}$, $n \geq 2$ be the tangle illustrated in Figure \ref{xntangle}. A basis of $P_n$ 
is given by
\begin{figure}[!h]
\begin{center}
\psfrag{1}{\huge $1$}
\psfrag{2}{\huge $2$}
\psfrag{3}{\huge $3$}
\psfrag{n-2}{\huge $n-2$}
\psfrag{n-1}{\huge $n-1$}
\psfrag{cdots}{\huge $\cdots$}
\resizebox{10.0cm}{!}{\includegraphics{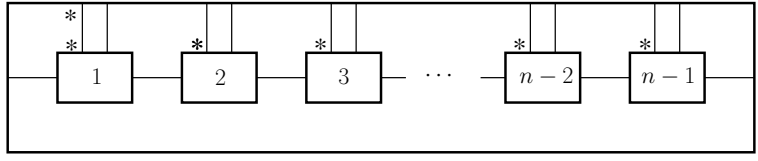}}
\end{center}
\caption{The tangle $X=X^n_{2,2,\cdots,2}$}
\label{xntangle}
\end{figure}
the set $\{Z_{X}(\gamma^1_{p_1q_1},\cdots,\gamma^{n-1}_{p_{n-1}q_{n-1}}) : \gamma^i \in \widehat{H^*}, 1 \leq p_i,q_i \leq d_{\gamma^i}\}$. In particular, $Z_X$
is an isomorphism.
\item[(4)] The relation in Figure \ref{deltarel} holds in $P(H)$ for any $\gamma \in \widehat{H^*}.$ \qed
\begin{figure}[!h]
\begin{center}
\psfrag{x1}{\huge $\gamma_{pq}$}
\psfrag{x2}{\huge $\gamma_{qt_k}^*$}
\psfrag{xn-1}{\huge $\gamma_{t_2t_1}^*$}
\psfrag{xn}{\huge $\gamma_{t_1p}^*$}
\psfrag{cdots}{\huge $\cdots$}
\psfrag{text}{\huge $= \sum_{t_1,\cdots,t_k}$}
\resizebox{10.0cm}{!}{\includegraphics{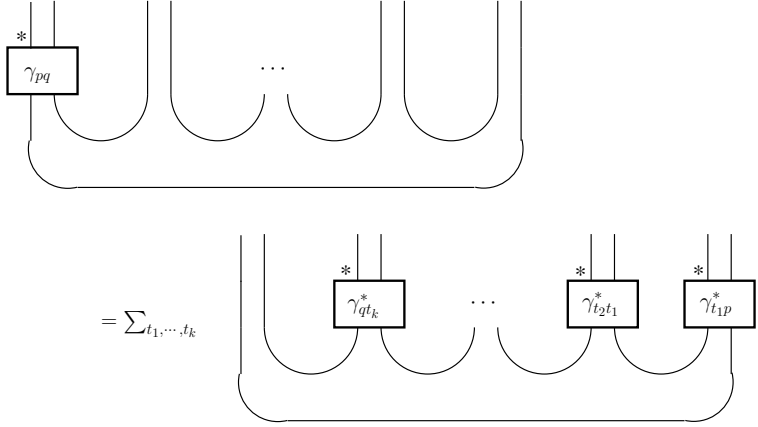}}
\end{center}
\caption{Useful relation in $P(H)$}
\label{deltarel}
\end{figure}
\end{itemize}
\end{proposition}

\section{Results from free probability theory}

The goal of this section is to give a very brief survey of free probability
theory and state the results that we will use in later sections.
We will use \cite{NcaSpc2006} and \cite{VclDykNca1992} as references.

\begin{definition}
An algebraic non-commutative
probability space consists of a unital algebra $A$ together with a linear
functional $\phi$ on $A$ such that $\phi(1) = 1$. 
It is said to be a $C^*$-algebraic probability space if $A$ is a $C^*$-algebra
and $\phi$ is a state, and to be a von Neumann algebraic probability
space if $A$ is a von Neumann algebra and $\phi$ is a normal state. 
\end{definition}

In this paper, all probability spaces we consider have tracial $\phi$.


%

\begin{definition}
If $(A,\phi)$ is a non-commutative probability space,
a family $\{A_i : i \in I\}$
of unital subalgebras of $A$ is said to be freely independent,
or simply free,
if for any positive integer $k$, indices $i_1,i_2,\cdots,i_k \in I$ such that
$i_1 \neq i_2, i_2 \neq i_3,\cdots, i_{k-1} \neq i_k$ and elements $a_t \in A_{i_t}$ with $\phi(a_t) = 0$ for $t = 1,2,\cdots,k$, the equality $\phi(a_1a_2\cdots a_t) =0$ holds. 
\end{definition}

In short, an alternating product
of centered elements is to be centered, with the obvious definitions.

\begin{definition}
If $\{(A_i,\phi_i) : i \in I\}$ is a family of algebraic non-commutative probability spaces, there is
a unique linear functional $\phi$ on the algebraic free product algebra $A = *_{i \in I} A_i$, such that
$\phi|_{A_i} = \phi_i$ and such that $\{A_i : i \in I\}$ (identified with
their images in $A$) is a freely independent family. The space
$(A,\phi)$ is said to be the free product of the family $\{(A_i,\phi_i) : i \in I\}$.
\end{definition}

There are notions of free products of $C^*$-algebraic
and von Neumann algebraic non-commutative
probability spaces  which require more work to define carefully and which we will
use without further explanation - see Chapter 1 of \cite{VclDykNca1992}.

In many contexts, it is important to decide whether a given
family of subalgebras of a non-commutative probability space
is a free family. For our purposes, the most convenient way to
do this is in terms of the free cumulants of the space which we will
now recall.

%


The lattice of non-crossing partitions plays a fundamental role
in the definition of free cumulants. Recall that for a totally ordered finite set $S$,
a partition $\pi$ of $S$ is said to be non-crossing if
whenever $i<j$ belong to a class of $\pi$ and $k<l$
belong to a different class of $\pi$, then it is not the case that
$k < i < l < j$ or $i < k < j < l$. The collection of non-crossing partitions of $S$, denoted
$NC(S)$, forms a lattice for the partial order defined by $\pi \geq  \rho$
if $\pi$ is coarser than $\rho$ or equivalently, if $\rho$ refines $\pi$.
The largest element of the lattice $NC(S)$  is denoted
$1_S$. Explicitly, $1_S = \{S\}$. If $S = [n] \stackrel{\mbox {\tiny {def}}}{=} \{1,2,\cdots,n\}$
for some $n \in {\mathbb N}$,  we will write
$NC(n)$ and $1_n$ for $NC(S)$ and $1_S$ respectively.

If $X$ is any set and $\{\phi_n : X^n \rightarrow {\mathbb C}\}_{n \in {\mathbb N}}$ is
a collection of functions, by the multiplicative extension of this collection, we will mean the collection of functions $\{\phi_\pi : X^n \rightarrow {\mathbb C}\}_{n \in {\mathbb N}, \pi \in NC(n)}$ defined by
$\phi_\pi(x^1,x^2,\cdots,x^n) = \prod_{C \in \pi} \phi_{|C|}(x^c : c \in C)$, where the arguments of each $\phi_{|C|}$ are listed with increasing
indices. Note that $\phi_n = \phi_{1_n}$. We now state a basic
combinatorial result (roughly equivalent to Proposition 10.21 of \cite{NcaSpc2006}) that we will refer to as M\"{o}bius inversion.
Let $\mu(\cdot,\cdot)$ be the M\"{o}bius function of the lattice $NC(n)$ - see Lecture 10 of \cite{NcaSpc2006}.

\begin{theorem}\label {Mobius}
Given two collections of functions $\{\phi_n : X^n \rightarrow {\mathbb C}\}_{n \in {\mathbb N}}$ and $\{\kappa_n : X^n \rightarrow {\mathbb C}\}_{n \in {\mathbb N}}$ extended multiplicatively, the following conditions are all equivalent:
\begin{itemize}
\item[(1)] $\phi_n = \sum_{\pi \in NC(n)} \kappa_\pi$ for each $n \in {\mathbb N}$.
\item[(2)] $\kappa_n = \sum_{\pi \in NC(n)} \mu(\pi,1_n) \phi_\pi$ for each $n \in {\mathbb N}$.
\item[(3)] $\phi_\tau = \sum_{\pi \in NC(n), \pi \leq \tau} \kappa_\pi$ for each $n \in {\mathbb N}, \tau \in NC(n)$.
\item[(4)] $\kappa_\tau = \sum_{\pi \in NC(n), \pi \leq \tau} \mu(\pi,\tau) \phi_\pi$ for each $n \in {\mathbb N}, \tau \in NC(n)$.
\end{itemize}
\end{theorem}

\begin{proof}[Sketch of Proof] Clearly $(3) \Rightarrow (1)$ and $(4) \Rightarrow (2)$ by taking $\tau = 1_n$. On the other hand, given $(2)$ and an arbitrary $\tau \in NC(n)$, we get:
\begin{eqnarray*}
\kappa_\tau(x^1,x^2,\cdots,x^n) &=& \prod_{C \in \tau} \kappa_{|C|}(x^c : c \in C)\\
               &=& \prod_{C \in \tau} \sum_{\pi_C \in NC(C)} \mu(\pi_C,1_C) \phi_{\pi_c}(x^c : c \in C)\\
               &=& \sum_{\pi \in NC(n), \pi \leq \tau} \mu(\pi,1_n) \phi_\pi(x^1,x^2,\cdots,x^n),
\end{eqnarray*}
where the last equality is a consequence of the natural bijection
between $\{\pi \in NC(n): \pi \leq \tau\}$ and collections $\{\{\pi_C \in NC(C)\}_{C \in \tau}\}$ given by $\pi = \cup_{C \in \tau} \pi_C$ under which
(i) $\phi_\pi(x^1,x^2,\cdots,x^n) = \prod_{C \in \tau} \phi_{\pi_C}(x^c: c \in C)$
and (ii) $\mu(\pi,1_n) = \prod_{C \in \tau} \mu(\pi_c,1_C)$.
This proves $(4)$ and so $(2) \Leftrightarrow (4)$.
An even easier proof shows that $(1) \Leftrightarrow (3)$.
Finally, $(3) \Leftrightarrow (4)$ by usual Mobius inversion in the
poset $NC(n)$.
\end{proof}

\begin{definition}
The free cumulants of a non-commutative probability space $(A,\phi)$
are the functions $\kappa_n : A^n \rightarrow {\mathbb C}$ 
associated as in Theorem \ref{Mobius} to the collection of functions
 $\{\phi_n : A^n \rightarrow {\mathbb C}\}_{n \in {\mathbb N}}$ defined by $\phi_n(a^1,\cdots,a^n) = \phi(a^1a^2\cdots a^n)$. 
\end{definition}
 
 The reason for their importance lies in the following theorem
of Speicher.

\begin{theorem}[Theorem 11.20 of \cite{NcaSpc2006}]\label{cumulant}
Let $(A,\phi)$ be a non-commutative probability space and $\{A_i : i \in I\}$ be a family of unital subalgebras of $A$
such that $A_i$ is generated as an algebra by $G_i \subseteq A_i$. This family is freely independent
iff for each positive integer $k$, indices $i_1,\cdots,i_k \in I$ that
are not all equal and elements $a_t \in G_{i_t}$ for $t = 1,2,\cdots,k$,
the equality $\kappa_k(a_1,a_2,\cdots,a_k) = 0$ holds. \qed
\end{theorem}

We also need a result of Nica and Speicher on `$R$-cyclic matrices'
- see Lecture 20 of
\cite{NcaSpc2006} - 
of a special type. Let $(A,\phi)$ be a non-commutative probability
space, $d \in {\mathbb N}$ and $(M_d(A),\phi^d)$ be the associated
matrix probability space where $\phi^d(X) = \frac{1}{d} \sum_i \phi(x_{ii})$ for $X = ((x_{ij})) \in M^d(A)$. 
Let $\kappa_*(\cdots)$ and $\kappa^d_*(\cdots)$ denote the free cumulants of $A$
and $M_d(A)$ respectively.

\begin{definition}
Call a matrix $X = ((x_{ij})) \in M_d(A)$ uniformly $R$-cyclic with determining sequence $\{\alpha_t \in {\mathbb C}\}_{t \in {\mathbb N}}$ if for any $i_1,j_1,i_2,j_2,\cdots,i_t,j_t \in \{1,2,\cdots,d\}$,
\begin{equation*}
 \kappa_t(x_{i_1,j_1},x_{i_2,j_2},\cdots,x_{i_t,j_t}) = 
 \left\{ 
\begin{array}{ll}
                   \alpha_t & \mbox{if } ~j_1 = i_2,j_2 = i_3,\cdots,j_{t-1} = i_t,j_t=i_1\\
                   0  & \mbox{otherwise}
        \end{array}
\right. ~.
 \end{equation*}
The adjective `uniform' refers to the fact that the cumulants are
independent of the indices $i_s,j_s$.
 \end{definition}

\begin{theorem}[Theorems 14.18 and 14.20 of \cite{NcaSpc2006}]\label{rcyclic}
Fix $X = ((x_{ij})) \in M_d(A)$. Let $A_1 = M_d({\mathbb C}) \subseteq M_d(A)$ and $A_2$ be the unital subalgebra of $M_d(A)$ generated
by $X$.
The following conditions are then equivalent:
\begin{itemize}
\item The matrix $X$ is uniformly $R$-cyclic with (some) determining sequence $\{\alpha_t\}_{t \in {\mathbb N}}.$
  \item $A_1$ and $A_2$ are free.
\end{itemize}
If these conditions hold, then $\kappa^d_t(X,X,\cdots,X) = 
d^{t-1} \alpha_t$.\qed
\end{theorem}


The results summarised so far have an algebraic/combinatorial
flavour. To get results about subfactors, we need some analytic 
input that is contained in the next few results. We use the following
notation and conventions. If $(A,\phi_A)$ and $(B,\phi_B)$ are non-commutative probability spaces and $0 < \alpha < 1$, by $\underset{\alpha}{A} \oplus
\underset{1-\alpha}{B}$, we will denote the non-commutative probability
space $(A \oplus B,\phi)$ where $\phi = \alpha \phi_A + (1-\alpha)\phi_B$. If $\alpha = 0$ (respectively $\alpha = 1$) then $\underset{\alpha}{A} \oplus
\underset{1-\alpha}{B}$ will denote $(B,\phi_B)$ (respectively $(A,\phi_A)$. If $A = LG$ is the von Neumann algebra of a countable group $G$, then we will regard
$A$ as a von Neumann algebraic tracial probability space with $\phi_A$ determined by $\phi_A(g) = \delta_{g1}$ for $g \in G$.
If $A$ is a finite factor, we regard $A$ as a von Neumann algebraic
probability space with $\phi_A = tr_A$ - the unique trace on $A$.


\begin{lemma}[Proposition 2.5.7 of \cite{VclDykNca1992}]\label{dc}
Let $(A,\phi)$ be a von Neumann algebraic non-commutative probability space and $\{A_i : i \in I\}$ be a family of unital *-subalgebras of $A$. Then $\{A_i : i \in I\}$ is a free family
iff $\{A_i^{\prime\prime} : i \in I\}$ is a free family.\qed
\end{lemma}

\begin{proposition}\label{poisson}
Let $(A,\phi)$ be a von Neumann algebraic non-commutative probability space with free cumulants $\kappa_n$ and
let $x \in A$ be a self-adjoint element such that $\kappa_n(x,x,\cdots,x) = \delta^{n-1}$ for a $\delta > 1$.
Let $B$ be the
von Neumann algebra generated by $x$ and set $\phi_B = \phi|_B$.
Then
$$
(B,\phi_B) \cong  
\underset{\ \ 1-\delta^{-1}}{{\mathbb C}} \oplus \underset{\ \ \delta^{-1}}{L{\mathbb Z}}.
$$
\end{proposition}

\begin{proof}
Recall that a self-adjoint element $a$ in a von Neumann algebraic
probability space is said to be
a free Poisson variable with rate $\lambda > 0$ and jump size $\alpha \in {\mathbb R}$ if $\kappa_t(a,a,\cdots,a) = \lambda \alpha^t$.
Thus our element $x$ is free Poisson
with rate $\delta^{-1}$ and jump-size $\delta$ and, by Proposition
12.11 of \cite{NcaSpc2006}, generates a 
von Neumann algebra isomorphic to $L^\infty({\mathbb R},\mu)$ where the
measure $\mu$ is of the form $(1-\delta^{-1})\nu_0 + \delta^{-1}\nu$ -
where $\nu_0$ is the point-mass at $0$ and $\nu$ is a probability measure
supported on an interval  $[a,b] \subseteq (0,\infty)$ that is mutually absolutely continuous
with respect to the Lebesgue measure. Under this isomorphism,
$\phi_B$ goes over to integration with respect to $\mu$.

Hence 
\begin{eqnarray*}
(B,\phi|_B) &\cong&
\underset{1-\delta^{-1}}{(L^\infty(\{0\},\nu_0),\int (\cdot)d\nu_0)} \oplus
\underset{\delta^{-1}}{((L^\infty([a,b],\nu),\int (\cdot)d\nu)}\\
&\cong&
\underset{1-\delta^{-1}}{({\mathbb C},id_{\mathbb C})} \oplus
\underset{\delta^{-1}}{((L^\infty(S^1,m),\int (\cdot)dm)}\\
&\cong& \underset{\ \ 1-\delta^{-1}}{{\mathbb C}} \oplus \underset{\ \ \delta^{-1}}{L{\mathbb Z}},
\end{eqnarray*}
where the last isomorphism uses the Fourier transform.
\end{proof}

Thus, Proposition \ref{poisson} determines the von Neumann algebraic
probability space generated by a 
free Poisson variable with rate $\delta^{-1}$ and jump size $\delta$. 

We now recall from \cite{Dyk1994} and \cite{Rdl1994} basic
properties of the interpolated free
group factors $LF(r)$ defined for $r > 1$. Set $LF(1) = L{\mathbb Z}$.
If $M$ is a finite factor and $\alpha > 0$, the $\alpha$-ampliation of $M$ (defined only for $\alpha$ being an integral multiple of ${\frac{1}{n}}$ if
$M$ is of type $I_n$)
denoted 
$M_\alpha$, - see
\cite{MrrNmn1943} - stands  for $pMp$ where $p \in M$ is
a projection of trace $\alpha$ if $\alpha <1$, for $M_n(M)$ if
$\alpha=n \in {\mathbb N}$, and satisfies $(M_{\alpha})_{\beta} \cong
M_{\alpha\beta}$ in general.

\begin{proposition}[Theorems 4.1 and 2.4 of \cite{Dyk1994}, Propositions 4.4 and 4.5 of \cite{Rdl1994}]\label{dykrdl}
Let $r,s > 1$ and $\alpha >0$.Then:
\begin{itemize}
\item[(1)] $LF(r)*LF(s) \cong
LF(r+s)$, and
\item[(2)] $LF(r)_\alpha \cong LF(\frac{r-1}{\alpha^2}+1).$\qed
\end{itemize}
\end{proposition}

The other analytic results we need are from \cite{Dyk1994} on computations of free products of tracial von Neumann algebraic
probability spaces.

\begin{proposition}[Proposition 1.7 of \cite{Dyk1994}]\label{dykprop}
Let $r,s \geq 1$ and $0 \leq \alpha,\beta \leq 1$. Then:
\begin{eqnarray*}
\lefteqn{
(\underset{1-\alpha}{{\mathbb C}} \oplus \underset{\alpha}{LF(r)})
*
(\underset{1-\beta}{{\mathbb C}} \oplus \underset{\beta}{LF(s)})
=}\\
& &
\left\{ 
\begin{array}{llr}
                   LF(r\alpha^2 + 2\alpha(1-\alpha) + s\beta^2 + 2\beta(1-\beta)) & {\text {if\ }} \alpha + \beta \geq 1 &  \\
                    \underset{1-\alpha -\beta}{{\mathbb C}} \oplus \underset{\alpha+\beta}{\underbrace{LF((\alpha+\beta)^{-2}(r\alpha^2+s\beta^2+4\alpha \beta))}} & {\text {if\ }} \alpha+\beta \leq 1. & \hskip.8in \Box
        \end{array} 
 \right.
\end{eqnarray*}
\end{proposition}

What we will actually use is the following corollary of Proposition \ref{dykprop}
which is easily proved by induction on $N$.

\begin{corollary}\label{dyk1} Let $\delta > 1$ and $N \in {\mathbb N}$. Then
$$
\hskip.8in
(\underset{\ \ 1-\delta^{-1}}{{\mathbb C}} \oplus \underset{\ \ \delta^{-1}}{L{\mathbb Z}})^{*N} =
\left\{ 
\begin{array}{llr}
                   LF(N(2\delta^{-1} - \delta^{-2})) & {\text {if\ }} N \geq \delta & \\
                    \underset{1-N\delta^{-1}}{{\mathbb C}} \oplus \underset{N\delta^{-1}}{LF(2-\frac{1}{N})} & {\text {if\ }} N \leq \delta. &
        \end{array}
\right.\hskip.45in \Box
$$
\end{corollary}

\begin{proposition}[Lemma 3.4 of \cite{Dyk1994}]\label{dyk2}
Let $r \geq 1$ and $0 \leq \alpha \leq 1$ and $d \in {\mathbb N}$. Then:
\begin{eqnarray*}
\lefteqn{
(\underset{1-\alpha}{{\mathbb C}} \oplus \underset{\alpha}{LF(r)})
*M_d({\mathbb C}) =
}
 \\
& &\hskip .3in
\left\{ 
\begin{array}{ll}
                   LF(r\alpha^2+2\alpha(1-\alpha)+1-d^{-2}) & {\text {if\ }} \alpha \geq d^{-2} \\
                    \underset{1-\alpha d^2}{M_d({\mathbb C})} \oplus \underset{\alpha d^2}{LF(rd^{-4}-2d^{-4}+1+d^{-2})} & {\text {if\ }} \alpha \leq d^{-2}.
        \end{array}
\right. \hskip1in \Box
\end{eqnarray*}
\end{proposition}

\begin{proposition}[Special case of Theorem 4.6 of \cite{Dyk1994}]\label{dyk3}
Let $A$ be a finite-dimensional von Neumann algebra
and $\phi$ be the normalised trace on $A$ in its left regular
representation (so that each central minimal projection of $A$
has trace $\frac{1}{n}$).
Suppose that $\frac{1}{n} \leq \alpha \leq 1$. Then,
$$
\hskip 1.3in
(\underset{1-\alpha}{{\mathbb C}} \oplus \underset{\alpha}{L{\mathbb Z}})
 * A \cong LF(2\alpha - \alpha^2 +1-\frac{1}{n}).\hskip 1.5in \Box
$$
\end{proposition}

\section{Determination of $M_1$}

Let $H$ be a finite dimensional Kac algebra of dimension $n > 1$
and let $P = P(H)$ be its planar algebra. Let $M_0 \subseteq
M_1 \subseteq \cdots$ be the tower of factors associated
to $P$ by the GJS-construction, so that $M_k$ is the von Neumann
algebra generated by $Gr_k(P)$ in the GNS-representation
afforded by $\tau_k$. Our goal in this section is to prove the following theorem.

\begin{theorem}\label{main1} Let $H$ be a finite dimensional Kac algebra of dimension $n > 1$, $P = P(H)$ be its planar algebra and $M_0 \subseteq
M_1 \subseteq \cdots$ be the tower of factors associated
with $P$ by the GJS-construction. Then, $M_1 \cong LF(2\sqrt{n}-1)$.
\end{theorem}

The strategy of proof is to find a free family $\{A(\gamma)\}_\gamma$of subalgebras of
$Gr_1(P)$ that generate it as an algebra (and hence also $M_1$ as a
von Neumann algebra), identify the von Neumann algebra $M(\gamma) =
A(\gamma)^{\prime\prime}$, and compute $*_\gamma M(\gamma)$, which is $M_1$.
To begin with, we determine the structure of $Gr_1(P)$.

Let $T(H) = \oplus_{n \geq 0} H^{\otimes n}$ be the tensor algebra
of the complex vector space $H$ regarded as a graded algebra
with $H^{\otimes n}$ being the degree $n$ piece. Define a $*$-structure on $T(H)$ by defining $(x^1 \otimes x^2 \otimes \cdots \otimes x^n)^* = S(x^n)^* \otimes \cdots \otimes S(x^2)^*
\otimes S(x^1)^*$, for $x^1,x^2,\cdots,x^n \in H$. Recall that
$S$ is the antipode of $H$ and corresponds  - see Theorem \ref{paha} -
to the rotation map $Z_R$ on $P_2$ (under the  
identification of $H$ with $P_2$).

\begin{proposition}\label{gr1p}
As graded $*$-algebras, $T(H)$ and $Gr_1(P)$ are isomorphic.
\end{proposition}

\begin{proof} 
Define a graded map from $T(H)$ to $Gr_1(P)$ by letting its
restriction to $H^{\otimes (n-1)} \subseteq T(H)$ be $Z_X$ where
$X = X^{n}_{2,2,\cdots,2}$ as defined in Figure \ref{xntangle}. This map
is easily verified to be a $*$-algebra isomorphism. Indeed, multiplicativity amounts to checking that with $M = M(1)_{m,n}^{m+n-1}$, $M\circ_{(1,2)}(X^m_{2,2,\cdots,2},X^n_{2,2,\cdots,2}) =
X^{m+n-1}_{2,2,\cdots,2}$, while $*$-preservation is seen to
follow from $D\circ X^* = \sigma(X)\circ_{(1,2,\cdots,n-1)}(R,R,\cdots,R)$ where, $D = D(1)^{n}_{n}$,
$X=X^{n}_{2,2,\cdots,2}$, $\sigma$ is the order reversing
involution of $\{1,2,\cdots,n-1\}$ and $\sigma(X)$ is the tangle
$X$ with $i^{th}$-internal box numbered $\sigma(i)$ for each $i$.
Both these tangle facts are seen to hold by drawing the appropriate
pictures.

Finally, it is seen from Proposition \ref{xprop}(3) that this map yields
an isomorphism, as desired.
\end{proof}




%
Note that Proposition \ref{gr1p} implies that $Gr_1(P) $ is generated
as a unital algebra by $P_2 \subseteq Gr_1(P)$. We now regard
$Gr_1(P)$ together with its trace $\tau_1 = \delta^{-1}Tr_1$ as a non-commutative probability space. Denoting the free cumulants by $\kappa_*(\cdots)$,
we wish to compute these explicitly on the generators. This can be done
in greater generality as in Proposition \ref{freecumulants}.

\begin{proposition}\label{freecumulants}
Let $P$ be any subfactor planar algebra of modulus $\delta$ that is irreducible (i.e., $P_1 \cong {\mathbb C}$), and $(Gr_1(P),\tau_1)$ be the GJS-probability space associated
to it (as in the preceding paragraph). 
If $x^1,\cdots,x^t \in P_2 \subseteq Gr_1(P)$, then
$\kappa_t(x^1,x^2,\cdots,x^t)$ is given by the tangle in Figure
\ref{fig:kappafig}. 
\end{proposition}

\begin{figure}[!h]
\begin{center}
\psfrag{epa}{\huge $\epsilon(a)$}
\psfrag{eq}{\huge $=$}
\psfrag{cdots}{\huge $\cdots$}
\psfrag{d}{\huge $\delta^{-1}$}
\psfrag{x1}{\huge $x^1$}
\psfrag{x2}{\huge $x^2$}
\psfrag{xn-1}{\huge $x^{t-1}$}
\psfrag{xn}{\huge $x^t$}
\psfrag{+}{\huge $+$}
\psfrag{del}{\huge $\delta$}
\resizebox{7.0cm}{!}{\includegraphics{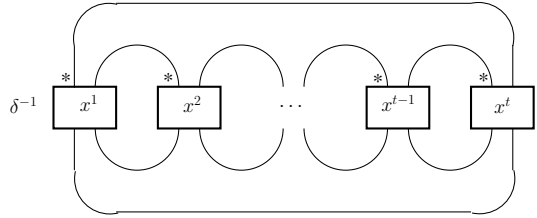}}
\end{center}
\caption{Identification of the free cumulants}
\label{fig:kappafig}
\end{figure}

Before proving this, we remind the reader of the well-known bijection
between non-crossing partitions and Temperley-Leib diagrams. We illustrate this in Figure \ref{tlncpcorr} with a
single example that should suffice.
\begin{figure}[!htb]
\psfrag{epa}{\huge $\epsilon(a)$}
\psfrag{eq}{\huge $=$}
\psfrag{cdots}{\huge $\cdots$}
\psfrag{text}{\huge $\leftrightarrow \ \ \ \ \ \{\{1,2,5\},\{3,4\},\{6\}\}$}
\psfrag{x1}{\huge $x^1$}
\psfrag{x2}{\huge $x^2$}
\psfrag{xn-1}{\huge $x^{t-1}$}
\psfrag{xn}{\huge $x^t$}
\psfrag{+}{\huge $+$}
\psfrag{del}{\huge $\delta$}
\resizebox{5.0cm}{!}{\includegraphics{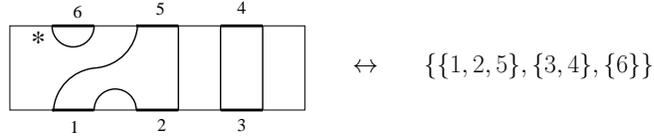}}
\caption{Bijection between TL-diagrams and non-crossing partitions}
\label{tlncpcorr}
\end{figure}
The Temperley-Lieb diagram on the left is to correspond to the non-crossing
partition on the right. 
Given a Temperley-Lieb diagram $T$, number the black boundary segments of the diagram anti-clockwise and take the partition corresponding
to the black regions to get the associated non-crossing partition $\pi_T$.
In the reverse direction,
denote the TL-diagram corresponding to a non-crossing partition $\pi$
by $TL(\pi)$ so that, for instance, $T_k = \sum_{\pi \in NC(k)}
TL(\pi)$, where, $T_k$ (recall from \S 1) is the sum of all the Temperley-Lieb elements
of $P_k$.

\begin{proof}[Proof of Proposition \ref{freecumulants}]


By definition of the product and trace in $Gr_1(P)$, we see that $\tau_1(x^1x^2\cdots x^t)$ is given by the expression in Figure \ref{fig:trfig}.
\begin{figure}[!h]
\begin{center}
\psfrag{tpi}{\huge $TL(\pi)$}
\psfrag{sum}{\huge $\sum_{\pi \in NC(t)}$}
\psfrag{delinphia}{\huge $\delta^{-1} \phi(a)$}
\psfrag{d}{\huge $\delta^{-1}$}
\psfrag{x1}{\huge $x^1$}
\psfrag{x2}{\huge $x^{2}$}
\psfrag{xn-1}{\huge $x^{t-1}$}
\psfrag{xn}{\huge $x^t$}
\psfrag{cdots}{\huge $\cdots$}
\psfrag{del}{\huge $\delta$}
\resizebox{8.0cm}{!}{\includegraphics{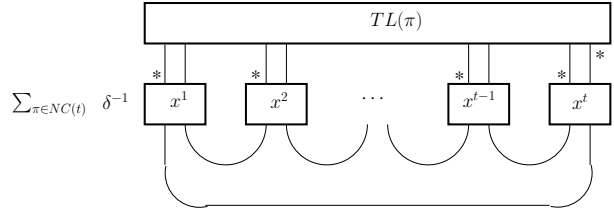}}
\end{center}
\caption{The trace of a product of elements of $P_2 \subseteq Gr_1(P)$}
\label{fig:trfig}
\end{figure}

We analyse the $\pi$-term of this sum. Since any non-crossing partition has a class that is an interval,
let $C$ be such a class of $\pi$ and suppose that $C = [k,l]$
where $1 \leq k \leq l \leq t$. 
The $\pi$-term then contains as a `sub-picture' the 1-tangle in Figure
\ref{fig:onetangle}.
 \begin{figure}[!h]
\begin{center}
\psfrag{tpi}{\huge $TL(\pi)$}
\psfrag{sum}{\huge $\sum_{\pi \in NC(k)}$}
\psfrag{delinphia}{\huge $\delta^{-1} \phi(a)$}
\psfrag{d}{\huge $\delta^{-1}$}
\psfrag{x1}{\huge $x^k$}
\psfrag{x2}{\huge $x^{k+1}$}
\psfrag{xn-1}{\huge $x^{l-1}$}
\psfrag{xn}{\huge $x^l$}
\psfrag{cdots}{\huge $\cdots$}
\psfrag{del}{\huge $\delta$}
\resizebox{6.0cm}{!}{\includegraphics{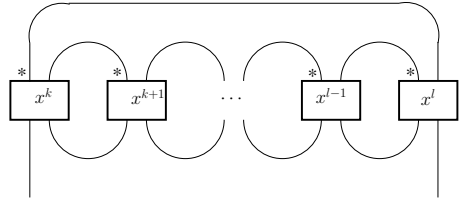}}
\end{center}
\caption{Sub-picture corresponding to the class $C$ of $\pi$}
\label{fig:onetangle}
\end{figure}
The irreducibility of the planar algebra $P$ implies that  this 1-tangle is 
a scalar multiple of $1_1$ (the unit element of $P_1$) the scalar
being given in Figure \ref{fig:scalar}.
\begin{figure}[!h]
\begin{center}
\psfrag{epa}{\huge $\epsilon(a)$}
\psfrag{eq}{\huge $=$}
\psfrag{cdots}{\huge $\cdots$}
\psfrag{d}{\huge $\delta^{-1}$}
\psfrag{x1}{\huge $x^k$}
\psfrag{x2}{\huge $x^{k+1}$}
\psfrag{xn-1}{\huge $x^{l-1}$}
\psfrag{xn}{\huge $x^l$}
\psfrag{+}{\huge $+$}
\psfrag{del}{\huge $\delta$}
\resizebox{7.0cm}{!}{\includegraphics{kappafig.eps}}
\end{center}
\caption{}
\label{fig:scalar}
\end{figure}

We may now peel off the next class of $\pi$ that is an interval and
proceed by induction to conclude that $\tau_1(x^1x^2\cdots x^t)$
is given by the expression in Figure \ref{fig:expression}, where we
write $C = \{i_1^C, i_2^C, \cdots ,i_{|C|}^C\}$.
\begin{figure}[!h]
\begin{center}
\psfrag{epa}{\huge $\epsilon(a)$}
\psfrag{eq}{\huge $=$}
\psfrag{sum}{\huge $\sum_{\pi \in NC(k)}$}
\psfrag{prod}{\huge $\prod_{C \in \pi}$}
\psfrag{cdots}{\huge $\cdots$}
\psfrag{d}{\huge $\delta^{-1}$}
\psfrag{x1}{\huge $x^{i_1^C}$}
\psfrag{x2}{\huge $x^{i_2^C}$}
\psfrag{xn-1}{\huge $x^{i_{|C|-1}^C}$}
\psfrag{xn}{\huge $x^{i_{|C|}^C}$}
\psfrag{+}{\huge $+$}
\psfrag{del}{\huge $\delta$}
\resizebox{10.0cm}{!}{\includegraphics{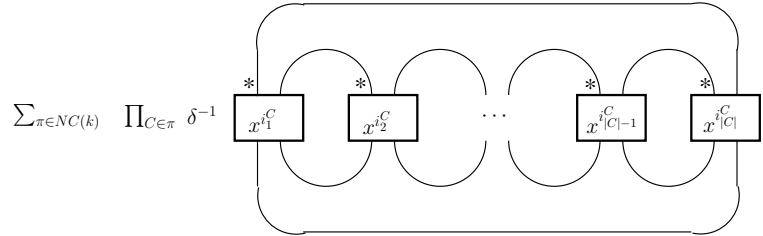}}
\end{center}
\caption{Expression for $\tau_1(x^1x^2\cdots x^t)$}
\label{fig:expression}
\end{figure}
Now, Mobius inversion (the implication $(1) \Rightarrow (2)$ of Theorem~\ref{Mobius})
yields the desired expression for $\kappa_t(x^1,x^2,\cdots,x^t)$.
\end{proof}

We extract a corollary of Proposition \ref{freecumulants} when $P = P(H)$. For $\gamma \in \widehat{H^*}$, let $A(\gamma)$
denote the subalgebra of $Gr_1(P)$ generated by $\gamma_{kl}
\in P_2 \subseteq Gr_1(P)$ for
$1 \leq k,l \leq d_\gamma$ and let $M(\gamma) = A(\gamma)^{\prime\prime} \subseteq M_1$. Let $X(\gamma) \in M_{d_\gamma}(M_1)$ be the $d_\gamma \times d_\gamma$ matrix $X(\gamma)
= ((\gamma_{kl}))$, and $\kappa^{d_\gamma}_*(\cdots)$ denote the free
cumulants of $M_{d_\gamma}(M_1)$.

\begin{corollary}\label{useful}
\itemize
\item[(1)] For each
$\gamma \in \widehat{H^*}$, the matrix $X(\gamma)$ is uniformly $R$-cyclic  with 
determining sequence $\{(\frac{\delta}{d_\gamma})^{t-1}\}_{t \in {\mathbb N}}$.
\item[(2)] The collection $\{M(\gamma)\}_{ \gamma \in \widehat{H^*}}$ is a 
free family in $M_1$.
\end{corollary}

\begin{proof}
The key calculation is that of the free cumulant $\kappa_t(\gamma^1_{i_1,j_1},\gamma^2_{i_2,j_2},\cdots,\gamma^t_{i_t,j_t})$ for $\gamma^1,\cdots,\gamma^t
\in \widehat{H^*}$, which, by Proposition \ref{freecumulants}
is given by the value of the tangle in Figure \ref{vanishes}.
\begin{figure}[!h]
\begin{center}
\psfrag{epa}{\huge $\epsilon(a)$}
\psfrag{eq}{\huge $=$}
\psfrag{cdots}{\huge $\cdots$}
\psfrag{d}{\huge $\delta^{-1}$}
\psfrag{x1}{\huge $\gamma^1_{i_1,j_1}$}
\psfrag{x2}{\huge $\gamma^2_{i_2,j_2}$}
\psfrag{xn-1}{\huge $\gamma^{t-1}_{i_{t-1},j_{t-1}}$}
\psfrag{xn}{\huge $\gamma^t_{i_t,j_t}$}
\psfrag{+}{\huge $+$}
\psfrag{del}{\huge $\delta$}
\resizebox{7.0cm}{!}{\includegraphics{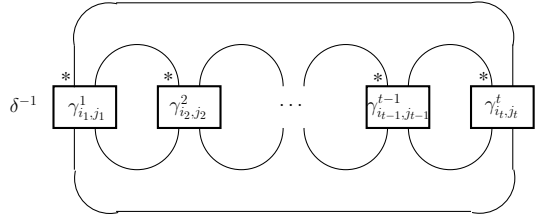}}
\end{center}
\caption{$\kappa_t(\gamma^1_{i_1,j_1},\gamma^2_{i_2,j_2},\cdots,\gamma^t_{i_t,j_t})$}
\label{vanishes}
\end{figure}
Judicious use of various parts of Proposition \ref{xprop} then shows that this vanishes
unless $\gamma^1 = \gamma^2 = \cdots = \gamma^t = \gamma$, say, and $j_1 = i_2$, $j_2 = i_3$,$\cdots$,$j_t = i_1$,
in which case it equals $(\frac{\delta}{d_\gamma})^{t-1}$.
This proves $(1)$
and, combined with Theorem \ref{cumulant} and Lemma \ref{dc}, yields $(2)$.
\end{proof}

The final hurdle to be crossed to prove Theorem \ref{main1} is the
determination of the structure of $M(\gamma)$; before getting to this,
we need an elementary fact.

\begin{lemma}\label{simple}
Suppose that $(A,\phi)$ is a von Neumann algebraic probability
space and $d\in {\mathbb N}$.
Assume that $$(M_d(A),\phi^d) \cong \underset{\alpha_1}{F_1} \oplus \underset{\alpha_2}{F_2} \oplus \cdots \oplus \underset{\alpha_k}{F_k},$$ where
the $F_i$ are all finite factors and $0 < \alpha_i \leq 1$ with $\sum_i \alpha_i = 1.$
Then, $$(A,\phi) \cong \underset{\alpha_1}{(F_1)_{\frac{1}{d}}} \oplus \underset{\alpha_2}{(F_2)_{\frac{1}{d}}} \oplus \cdots \oplus \underset{\alpha_k}{(F_k)_{\frac{1}{d}}}.$$
\end{lemma}

\begin{proof} 

%

%
Observe first that the direct sum decomposition of
the non-commutative probability space $(M_d(A),\phi)$ is unique
in the sense that if 
$$(M_d(A),\phi^d) \cong \underset{\alpha^\prime_1}{F^\prime_1} \oplus \underset{\alpha^\prime_2}{F^\prime_2} \oplus \cdots \oplus \underset{\alpha^\prime_k}{F^\prime_{k^\prime}},$$
is another such decomposition, then $k=k^\prime$ and, after a
rearrangement, $F_i \cong F^\prime_i$ and $\alpha_i = \alpha^\prime_i$.
To see this, let $\{e_1,\cdots,e_k\}$ be the set of minimal central
projections of $F_1 \oplus F_2 \oplus \cdots \oplus F_k$ and 
$\{e_1^\prime,\cdots,e_{k^\prime}^\prime\}$ be the corresponding set for $F_1^\prime \oplus F_2^\prime \oplus \cdots \oplus F_{k^\prime}^\prime$.
 The trace preserving
isomorphism between $F_1 \oplus F_2 \oplus \cdots \oplus F_k$ and 
$F_1^\prime \oplus F_2^\prime \oplus \cdots \oplus F_{k^\prime}^\prime$ induces a bijection between these sets,
so that $k = k^\prime$ and we may assume after rearrangement that
$e_i$ corresponds to $e^\prime_i$.
Further, the quotient of $F_1 \oplus F_2 \oplus \cdots \oplus F_k$ by $1-e_i$, which is $F_i$, is isomorphic to
the quotient of $F_1^\prime \oplus F_2^\prime \oplus \cdots \oplus F_k^\prime$ by $1-e^\prime_i$ which is $F^\prime_i$.
Finally $\alpha_i = \alpha^\prime_i$ since these are the traces of
$e_i$ and $e_i^\prime$ respectively and the isomorphism is trace
preserving.

Now, since $Z(A)\cong Z(M_d(A))$ which is $k$-dimensional, it follows that $A$ is isomorphic
to a direct sum of $k$ factors.
Suppose that $A \cong \tilde{F_1} \oplus \cdots \oplus \tilde{F_k}$
for factors $\tilde{F_i}$. 
Since $\phi^d$ is tracial and faithful (by the assumed strict
positivity of the $\alpha_i$'s), so is $\phi$ and so $(A,\phi) \cong
\underset{\beta_1}{\tilde{F_1}} \oplus \cdots \oplus \underset{\beta_k}{\tilde{F_k}}$ for some $0 < \beta_i \leq 1$ and therefore
$(M_d(A),\phi^d) \cong \underset{\beta_1}{M_d(\tilde{F_1})} \oplus \cdots \oplus \underset{\beta_k}{M_d(\tilde{F_k})}.$
By the observation made at the start of this proof, we may assume that $M_d(\tilde{F_i})
\cong F_i$ and that $\beta_i = \alpha_i$. Therefore,
$$(A,\phi) \cong \underset{\alpha_1}{(F_1)_{\frac{1}{d}}} \oplus \underset{\alpha_2}{(F_2)_{\frac{1}{d}}} \oplus \cdots \oplus \underset{\alpha_k}{(F_k)_{\frac{1}{d}}},$$
concluding the proof.
\end{proof}

\begin{proposition}\label{mgamma}
For $\gamma \in \widehat{H^*}$, let $\tau_\gamma = \tau_1|_{M(\gamma)}$. Then
$$
(M(\gamma),\tau_\gamma) \cong 
(\underset{1-\delta^{-1}}{\mathbb C} \oplus 
\underset{\delta^{-1}}{L{\mathbb Z}})^{*\ d_\gamma^2}.
$$
\end{proposition}



\begin{proof}
By definition, $M(\gamma)$ is the von Neumann algebraic probability subspace
of $(M_1,\tau_1)$
generated by the entries of $X(\gamma) \in M_{d_\gamma}(M_1)$. It follows that
$M_{d_\gamma}(M(\gamma))$ is the von Neumann algebraic probability subspace
of  $(M_{d_\gamma}(M_1),\tau_1^{d_\gamma})$ generated by $X(\gamma)$
and $M_{d_\gamma}({\mathbb C})$.

Notice now that although $\gamma_{kl}^* = S\gamma_{lk}$ (in $P(H)$) ,
we see from the definitions that $\gamma_{kl}^\dagger = \gamma_{lk}$
(in $Gr_1(P)$) and consequently $X(\gamma) \in M_{d_\gamma}(M_1)$ is self-adjoint.
By Corollary \ref{useful}(1), the matrix $X(\gamma)$ is uniformly $R$-cyclic with determining
sequence $\{(\frac{\delta}{d_\gamma})^{t-1}\}_{t \in {\mathbb N}}$; Theorem
\ref{rcyclic} now implies that $X(\gamma)$ is a free Poisson variable with rate $\delta^{-1}$ and jump size $\delta$ and so the von Neumann algebraic probability space that
it generates is $\underset{\ \ 1-\delta^{-1}}{{\mathbb C}} \oplus \underset{\ \ \delta^{-1}}{L{\mathbb Z}}$ by Proposition \ref{poisson}.

By Theorem \ref{rcyclic} again and Lemma \ref{dc}, the von Neumann algebraic
probability spaces generated by $X(\gamma)$ and $M_{d_\gamma}({\mathbb C})$
are free in $M_{d_\gamma}(M_1)$ and therefore
\begin{eqnarray*}
(M_{d_\gamma}(M(\gamma)),\tau_\gamma^{d_\gamma}) &\cong&
(\underset{1-\delta^{-1}}{{\mathbb C}} \oplus \underset{\ \ \delta^{-1}}{L{\mathbb Z}}) * M_{d_\gamma}({\mathbb C})\\ &\cong&
\left\{ 
\begin{array}{ll}
                   LF(2\delta^{-1}-\delta^{-2}+1-{d_\gamma}^{-2}) & {\text {if\ }} \delta^{-1} \geq {d_\gamma}^{-2} \\
                    \underset{1-\delta^{-1} {d_\gamma}^2}{M_{d_\gamma}({\mathbb C})} \oplus \underset{\delta^{-1} {d_\gamma}^2}{LF({-d_\gamma}^{-4}+1+{d_\gamma}^{-2})} & {\text {if\ }} \delta^{-1} \leq {d_\gamma}^{-2}
        \end{array}
\right. 
\end{eqnarray*}
where the last isomorphism appeals to Proposition \ref{dyk2}.
Now Lemma \ref{simple} and Proposition \ref{dykrdl} show that
$$
(M(\gamma)),\tau_\gamma) \cong
\left\{ 
\begin{array}{ll}
                   LF(2\delta^{-1}d_\gamma^2-\delta^{-2}{d_\gamma}^{2}) & {\text {if\ }} \delta^{-1} \geq {d_\gamma}^{-2} \\
                    \underset{1-\delta^{-1} {d_\gamma}^2}{\mathbb C} \oplus \underset{\delta^{-1} {d_\gamma}^2}{LF(2-d_\gamma^{-2})} & {\text {if\ }} \delta^{-1} \leq {d_\gamma}^{-2}
        \end{array}
\right. 
$$
Finally, an application of Corollary \ref{dyk1} with $N = d_\gamma^2$ yields the desired result.
\end{proof}


We conclude this section with the proof of its main result.

\begin{proof}[Proof of Theorem \ref{main1}]
Since the family $\{M(\gamma)\}_{\gamma \in \widehat{H^*}}$ is free
in $M_1$ and generates it as a von Neumann algebra, 
\begin{eqnarray*}
M_1 &\cong& *_{\gamma \in \widehat{H^*}}M(\gamma)\\
&\cong& *_{\gamma \in \widehat{H^*}}(\underset{1-\delta^{-1}}{\mathbb C} \oplus 
\underset{\delta^{-1}}{L{\mathbb Z}})^{*\ d_\gamma^2}\\
&\cong& (\underset{1-\delta^{-1}}{\mathbb C} \oplus 
\underset{\delta^{-1}}{L{\mathbb Z}})^{*\ n}\\
&\cong& LF(2\sqrt{n}-1),
\end{eqnarray*}
where the second isomorphism follows from Proposition \ref{mgamma}
and the last isomorphism from Corollary \ref{dyk1}.
\end{proof}

\section{Determination of $M_2$}

The main result of this section is the identification of $M_2$
as an interpolated free group factor. The strategy of proof
is similar to that of the last section. We determine a pair
of subalgebras of $Gr_2(P)$ that are free and generate it
and compute the free product of the generated von Neumann
algebras to determine $M_2$.

\begin{theorem}\label{main2}
Let $H$ be a finite dimensional Kac algebra of dimension $n > 1$, $P = P(H)$ be its planar algebra and $M_0 \subseteq
M_1 \subseteq \cdots$ be the tower of factors associated
to $P$ by the GJS-construction. Then, $M_2 \cong LF(\frac{2}{\sqrt{n}} - \frac{2}{n} + 1)$.
\end{theorem}

The first step is to determine the structure of $Gr_2(P)$.
The graded $*$-algebra $T(H)$ admits an action by the Kac algebra $H$ defined by $\alpha_a(x^1 \otimes x^2 \otimes \cdots \otimes x^t)
= a_1x^1 \otimes \cdots \otimes a_tx^t$ for $a \in H$ and
$x^1 \otimes x^2 \otimes \cdots \otimes x^t \in H^{\otimes t} \subseteq T(H)$ (where
we use the notation $a_1 \otimes a_2 \otimes \cdots \otimes a_t$ for
the iterated coproduct $\Delta^t(a)$).
We may form the crossed-product algebra $T(H) \rtimes_\alpha H$
and introduce a grading on it by declaring that $deg(w \rtimes a) =
deg(w)$ for any $a \in H$ and homogeneous $w \in T(H)$.
The natural inclusion $T(H) \subseteq T(H) \rtimes_\alpha H$ is a map 
of graded $*$-algebras.

\begin{proposition}\label{gr2p} The algebras $T(H) \rtimes_\alpha H$
and $Gr_2(P)$ are isomorphic as graded $*$-algebras by
an isomorphism that extends the isomorphism from
$T(H)$ to $Gr_1(P)$.
\end{proposition}

\begin{proof} Define $\theta : T(H) \rtimes_\alpha H \rightarrow Gr_2(P)$ by
letting $\theta(x^1 \otimes x^2 \otimes \cdots \otimes x^t \rtimes a)$
be given by the tangle in Figure \ref{imtheta}.
\begin{figure}[!h]
\psfrag{x1}{\huge $x^1$}
\psfrag{x2}{\huge $x^2$}
\psfrag{xt}{\huge $x^t$}
\psfrag{9}{\huge $X^9$}
\psfrag{a}{\huge $a$}
\psfrag{cdots}{\huge $\cdots$}
\psfrag{13}{\huge $X^{13}$}
\psfrag{15}{\huge $X^{15}$}
\psfrag{16}{\huge $X^{16}$}
\resizebox{8.0cm}{!}{\includegraphics{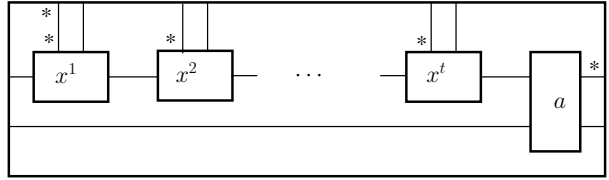}}
\caption{Definition of $\theta$}\label{imtheta}
\end{figure}
It is a straightforward consequence of the definitions that
the restriction of $\theta$ to $T(H)$ is the composition of the
isomorphism of $T(H)$ 
with $Gr_1(P)$ and the inclusion of $Gr_1(P)$ into $Gr_2(P)$, while
the restriction of $\theta$ to the acting $H$ is the natural
isomorphism of $H$ with 
$P_2 \subseteq Gr_2(P)$.
Also $\theta$ is a linear isomorphism since for each $t$, the tangle in Figure \ref{imtheta}
is just $X=X^{t+2}_{2,2,\cdots,2}$ - see Proposition \ref{xprop}(3) -
redrawn slightly differently. 
The crux of the verification of multiplicativity of $\theta$ is seen to reduce to
the equality asserted in Figure \ref{crux}, which is a consequence of the relations
in $P(H)$.
\begin{figure}[!h]
\psfrag{x1}{\huge $x^1$}
\psfrag{x2}{\huge $x^2$}
\psfrag{xt}{\huge $x^t$}
\psfrag{9}{\huge $X^9$}
\psfrag{a}{\huge $a$}
\psfrag{cdots}{\huge $\cdots$}
\psfrag{y1}{\huge $a_1x^1$}
\psfrag{y2}{\huge $a_2x^2$}
\psfrag{yt}{\huge $a_tx^t$}
\psfrag{=}{\huge $=$}
\psfrag{b}{\huge $a_{t+1}$}
\resizebox{8.0cm}{!}{\includegraphics{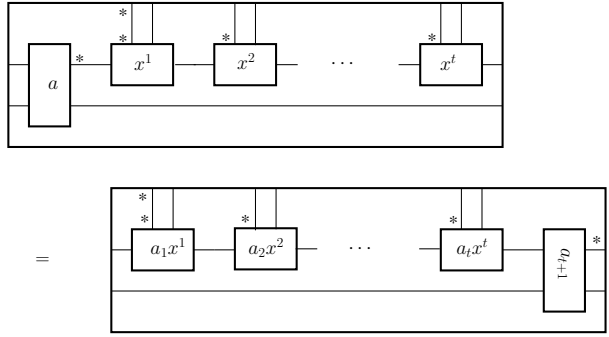}}
\caption{Multiplicativity of $\theta$}\label{crux}
\end{figure}
Finally, $\theta$ preserves $*$ since it does so on $T(H)$ (by
Proposition \ref{gr1p}) and on the acting $H$ (clearly!)
and is multiplicative.
\end{proof}

Since the crossed product algebra $T(H) \rtimes_\alpha H$ is clearly
generated by the two patent copies of $H$, it follows from Proposition
\ref{gr2p} that $Gr_2(P)$ is generated by $P_2 \subseteq Gr_2(P)$
and by the image of $P_2 \subseteq Gr_1(P)$ in $P_3 \subseteq Gr_2(P)$. We will require the following sharpening of this result.
Throughout this section we will denote the  image of $1 \in P_2 \subseteq Gr_1(P)$ in $P_3 \subseteq Gr_2(P)$ by $X$ and note that
pictorially, it is shown in the figure below.
\begin{figure}[!h]
\resizebox{1.0cm}{!}{\includegraphics{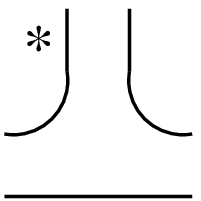}}
\end{figure}

\begin{proposition}\label{gengr2}
The algebra $Gr_2(P)$ is generated by $P_2 \subseteq Gr_2(P)$
and $X \in P_3 \subseteq Gr_2(P)$.
\end{proposition}

\begin{proof} From (the sentence immediately following) Proposition
  \ref{gr2p}, it suffices to verify that the image of $P_2 \subseteq
  Gr_1(P)$ in $P_3 \subseteq Gr_2(P)$ 
is contained in the subalgebra of $Gr_2(P)$ generated by $P_2$
and $X$. 
However for any $a,b \in P_2 \subseteq Gr_2(P)$, notice that $aXb$ is
given by the tangle in Figure \ref{axb}.
\begin{figure}[!htb]
\psfrag{a}{\huge $a$}
\psfrag{b}{\huge $b$}
\resizebox{3.0cm}{!}{\includegraphics{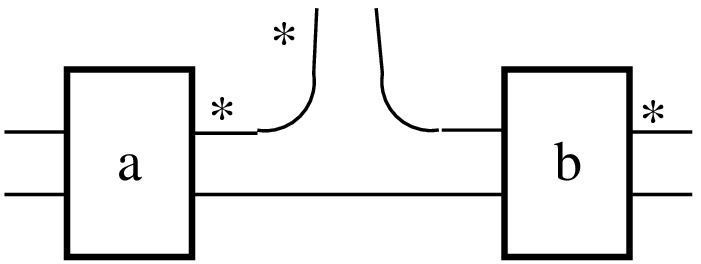}}
\caption{}\label{axb}
\end{figure}
Elements of this kind are easily verified to span the whole of $P_3$ using the depth 2
property of $P(H)$.
\end{proof}

The main combinatorial fact underlying the determination of $M_2$
is that the algebra $P_2$ and the algebra generated by $X$ are
free in it, which is what we will establish next. Recall that
$\tau_2 = \delta^{-2}Tr_2$ is a normalised trace on $Gr_2(P)$.
We will denote the associated free cumulants by $\kappa_*(\cdots)$.

Note that $(Gr_1(P),\tau_1)$ is a non-commutative probability subspace of
$(Gr_2(P),\tau_2)$ and so the free cumulants of $X$ in $Gr_2(P)$
are the same as those of $1 \in P_2 \subseteq Gr_1(P)$.
Since $1 = triv_{11}$ where $triv \in \widehat{H^*}$ is the
trivial representation of $H^*$, it follows - from Corollary \ref{useful} - that $\kappa_t(X,X,\cdots,X) = \delta^{t-1}$, or equivalently, that $X$ is free Poisson with rate $\delta^{-1}$
and jump size $\delta$.

\begin{proposition}\label{xp2free}
The algebra generated by $X$ and the algebra $P_2 \subseteq Gr_2(P)$
are free in the non-commutative probability space $(Gr_2(P),\tau_2)$.
\end{proposition}

\begin{proof}

Consider the problem of calculating $Tr_2(X^1X^2\cdots X^t)$ where each
$X^i \in P_2 \cup \{X\}$. Let $D = \{i \in [t]: X^i = X\}$ and
$E = \{i \in [t]: X^i \in P_2\}$ so that these are complementary
sets in $[t]$. 

We illustrate with an example. 
Suppose 
$t=16$ and $D = \{1,3,4,5,8,12,14,15\}$ so that $E =
\{2,6,7,9,10,11,13,16\}$. The product $\prod_{i=1}^{15} X^i$ in
$Gr_2(P)$ is is given by the
tangle in Figure \ref{chain}
\begin{figure}[!h]
\psfrag{2}{\huge $X^2$}
\psfrag{6}{\huge $X^6$}
\psfrag{7}{\huge $X^7$}
\psfrag{9}{\huge $X^9$}
\psfrag{10}{\huge $X^{10}$}
\psfrag{11}{\huge $X^{11}$}
\psfrag{13}{\huge $X^{13}$}
\psfrag{15}{\huge $X^{15}$}
\psfrag{16}{\huge $X^{16}$}
\resizebox{12.0cm}{!}{\includegraphics{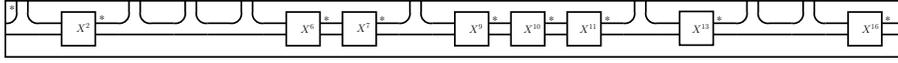}}
\caption{$\prod_{i=1}^{15} X^i$}\label{chain}
\end{figure}
and its trace is given by the sum over all $\pi \in NC(D)$ of the tangle
in Figure \ref{trchain}.
\begin{figure}[!h]
\psfrag{2}{\huge $X^2$}
\psfrag{6}{\huge $X^6$}
\psfrag{7}{\huge $X^7$}
\psfrag{9}{\huge $X^9$}
\psfrag{10}{\huge $X^{10}$}
\psfrag{11}{\huge $X^{11}$}
\psfrag{13}{\huge $X^{13}$}
\psfrag{15}{\huge $X^{15}$}
\psfrag{16}{\huge $X^{16}$}
\psfrag{text}{\huge $TL(\pi)$}
\resizebox{12.0cm}{!}{\includegraphics{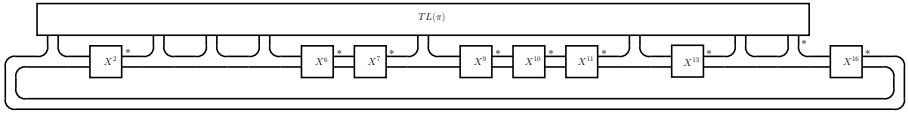}}
\caption{The $\pi$-term of $Tr_2(\prod_{i=1}^{15} X^i)$}\label{trchain}
\end{figure}
We will fix a $\pi \in NC(D)$ and analyse the $\pi$-term of the sum.
Again, an illustrative example will help. So we consider $\pi = 
\{\{1,5\},\{3,4\},\{8,14,15\},\{12\}\}$. Then the $\pi$-term is illustrated
in Figure \ref{piterm}.
\begin{figure}[!h]
\psfrag{2}{\huge $X^2$}
\psfrag{6}{\huge $X^6$}
\psfrag{7}{\huge $X^7$}
\psfrag{9}{\huge $X^9$}
\psfrag{10}{\huge $X^{10}$}
\psfrag{11}{\huge $X^{11}$}
\psfrag{13}{\huge  $X^{13}$}
\psfrag{15}{\huge $X^{15}$}
\psfrag{16}{\huge $X^{16}$}
\resizebox{12.0cm}{!}{\includegraphics{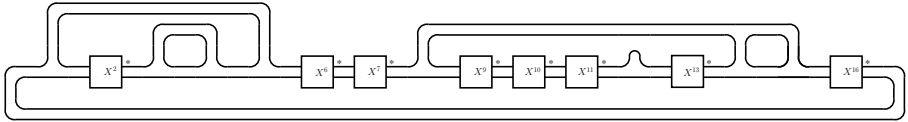}}
\caption{The $\{\{1,5\},\{3,4\},\{8,14,15\},\{12\}\}$-term of $Tr_2(\prod_{i=1}^{15} X^i)$}\label{piterm}
\end{figure}

Note that the $\pi$-term has several floating loops, each contributing
a multiplicative factor of $\delta$. 
Now remove the floating loops and the innermost string connecting all
the boxes $X^i$ for $i \in E$ to get Figure \ref{disjointed}.
\begin{figure}[!h]
\psfrag{2}{\huge $X^2$}
\psfrag{6}{\huge $X^6$}
\psfrag{7}{\huge $X^7$}
\psfrag{9}{\huge $X^9$}
\psfrag{10}{\huge $X^{10}$}
\psfrag{11}{\huge $X^{11}$}
\psfrag{13}{\huge $X^{13}$}
\psfrag{15}{\huge $X^{15}$}
\psfrag{16}{\huge $X^{16}$}
\resizebox{12.0cm}{!}{\includegraphics{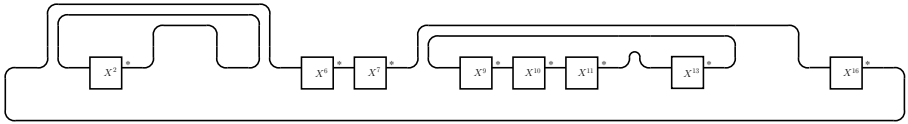}}
\caption{Disconnecting Figure \ref{piterm}}\label{disjointed}
\end{figure}

There are several connected components, each of which loops some
of the boxes $X^i$, $i \in E$ together and so defines a partition of $E$. 
Denote this partition by $\tilde{\pi}$.
A little thought should convince the reader that $\pi \cup \tilde{\pi}$
is a non-crossing partition of $[t]$ and that $\tilde{\pi}$ is 
coarser than any partition of $E$ with this property.

In our example, $\tilde{\pi} =
\{\{2\},\{6,7,16\},\{9,10,11,13\}\}$. By irreducibility of the planar
algebra $P$, any class, say $C$, of $\tilde{\pi}$ contributes a
multiplicative factor of $\delta \phi(\prod_{c \in C} X^c)$ (where the
product is taken with the $X^c$ listed in increasing order) to the
$\pi$-term. It follows that the $\pi$-term evaluates to
$\delta^{N(\pi)} \phi_{\tilde{\pi}}(X^e : e \in E)$, where $N(\pi)$ is
the number of loops in the figure obtained from Figure \ref{piterm} by
replacing all the $X^i, i \in E$ by $1_2 \in P_2$.  This latter figure
is shown below.

\begin{figure}[!h]
\psfrag{2}{\huge $X^2$}
\psfrag{6}{\huge $X^6$}
\psfrag{7}{\huge $X^7$}
\psfrag{9}{\huge $X^9$}
\psfrag{10}{\huge $X^{10}$}
\psfrag{11}{\huge $X^{11}$}
\psfrag{13}{\huge $X^{13}$}
\psfrag{15}{\huge $X^{15}$}
\psfrag{16}{\huge $X^{16}$}
\resizebox{12.0cm}{!}{\includegraphics{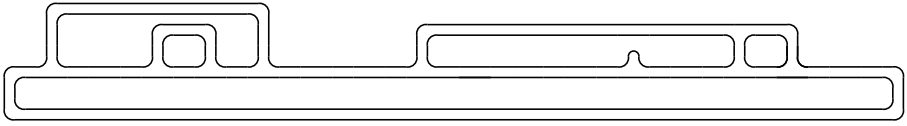}}
\caption{Replacing all $X^i$ in Figure \ref{piterm} by $1_2$}\label{loops}
\end{figure}

Therefore $Tr_2(X^1X^2\cdots X^t) = 
\sum_{\pi \in NC(D)} \delta^{N(\pi)} \phi_{\tilde{\pi}}(X^e : e \in E)$ and hence:

%



\begin{eqnarray*}
\tau_2(X^1X^2\cdots X^t) &=& \sum_{\pi \in NC(D)} \delta^{N(\pi)-2} \phi_{\tilde{\pi}}(X^e : e \in E)\\
&=&
\sum_{\pi \in NC(D)} \delta^{|D|-|\pi|} \phi_{\tilde{\pi}}(X^e : e \in E)\\
&=& \sum_{\pi \in NC(D)} \delta^{|D|-|\pi|} \left(\sum_{\rho \in NC(E),\rho \leq \tilde{\pi}} \kappa_\rho(X^e : e \in E)\right)
\end{eqnarray*}
where the second equality follows from Proposition \ref{euler} below,
and the third equality is by $(3)$ of Theorem~\ref{Mobius}. 

We now assert that 
$$
\tau_2(X^1X^2\cdots X^t) = \sum_{\lambda \in NC(t)} \tilde{\kappa}_\lambda(X^1,\cdots,X^t),
$$
where $\tilde{\kappa}_\lambda$ is the multiplicative extension of
$\{\tilde{\kappa}_t : (P_2 \cup \{X\})^t \rightarrow {\mathbb C}\}_{t \in {\mathbb N}}$
defined by
\begin{equation*}
\tilde{\kappa}_t(X^1,\cdots,X^t) = 
\left\{ 
\begin{array}{ll}
                   \delta^{t-1} & {\text {if\ all\ }} X^i = X \\
                    \kappa_t(X^1,\cdots,X^t) & {\text {if\ all\ }} X^i \in P_2\\
                    0 & {\text {otherwise.}}
        \end{array}
\right.
\end{equation*}
To prove this assertion, note that the only $\lambda \in NC(t)$ that contribute to the sum are those 
of the form $\pi \cup \rho$ where $\pi \in NC(D)$, $\rho \in NC(E)$ and
$\rho \leq \tilde{\pi}$,
and the corresponding term is exactly $\delta^{|D|-|\pi|} \kappa_\rho(X^e : e \in E)$.

But now, M\"{o}bius inversion implies that $\tilde{\kappa}_{\pi} = \kappa_{\pi}$ and Theorem \ref{cumulant} then proves the desired freeness
of $P_2$ and the algebra generated by $X$ in $Gr_2(P)$.
\end{proof}

\begin{proposition}\label{euler}
Let $n \in {\mathbb N}$ and $\pi \in NC(n)$.
By $L(\pi)$, we will denote the $0$-tangle in Figure \ref{lpi}. Let
$N(\pi)$ be the
number of loops in $L(\pi)$.
Then, $N(\pi) -2 = n - |\pi|$.
\begin{figure}[!h]
\psfrag{tlpi}{\huge $TL(\pi)$}
\psfrag{cdots}{\huge $\cdots$}
\psfrag{7}{\huge $X^7$}
\psfrag{9}{\huge $X^9$}
\psfrag{10}{\huge $X^{10}$}
\psfrag{11}{\huge $X^{11}$}
\psfrag{13}{\huge $X^{13}$}
\psfrag{15}{\huge $X^{15}$}
\psfrag{16}{\huge $X^{16}$}
\resizebox{6.0cm}{!}{\includegraphics{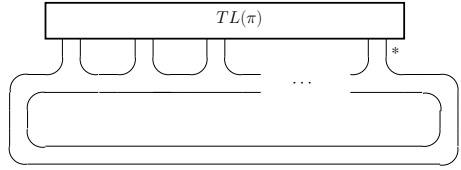}}
\caption{The $0$-tangle $L(\pi)$}\label{lpi}
\end{figure}
\end{proposition}

\begin{proof} 
The proof is an induction on the number of classes $|\pi|$ of
$\pi$. The basis case $|\pi| =1$ being easily proved, we consider the
case $|\pi|>1$. Since any non-crossing partition has a class that
is an interval, let $C= [k,l], 1 \leq k \leq l \leq n$ be such a class of $\pi$ and let $S$ denote
the complement of $C$ in $[n]$.
In a `neighbourhood' of $C$, the tangle $L(\pi)$ looks as in Figure \ref{nbd}.
\begin{figure}[!h]
\psfrag{tlpi}{\huge $TL(\pi)$}
\psfrag{cdots}{\huge $\cdots$}
\psfrag{k}{\Large $k$}
\psfrag{k+1}{\Large $k+1$}
\psfrag{k+2}{\Large $k+2$}
\psfrag{l-1}{\Large $l-1$}
\psfrag{l}{\Large $l$}
\psfrag{15}{\huge $X^{15}$}
\psfrag{16}{\huge $X^{16}$}
\resizebox{5.0cm}{!}{\includegraphics{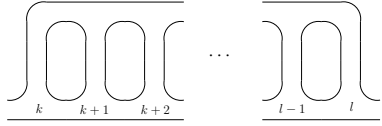}}
\caption{A `neighbourhood' of the class $C = [k,l]$ of $\pi$}\label{nbd}
\end{figure}
After removing the $l-k=|C|-1$ loops between $k$ and $l$,
it should be clear that what remains is the tangle $L(\pi_S)$
where $\pi_S = \pi|_S$. 
Hence, $N(\pi_S) = N(\pi) - (|C|-1)$ while $|\pi_S| = |\pi|-1$ and
$|S| = n - |C|$. The proof is complete by induction.
\end{proof}

We can now identify the factor $M_2$ as an interpolated free group
factor.

\begin{proof}[Proof of Theorem \ref{main2}]
By Proposition \ref{gengr2}, the algebra $Gr_2(P)$ is generated
by $P_2$ and $X \in P_3$ and so the factor $M_2$ is generated
as a von Neumann algebra by these. Since $X$ is free Poisson with
rate $\delta^{-1}$ and jump size $\delta$, the von Neumann
algebra it generates is isomorphic to 
$\underset{\ \ 1-\delta^{-1}}{\mathbb C} \oplus 
\underset{\delta^{-1}}{L{\mathbb Z}}$. Since this von Neumann algebra
and $P_2$ are free in $M_2$ by Proposition \ref{xp2free} and Lemma \ref{dc}, it follows that 
\begin{eqnarray*}
M_2 &\cong& (\underset{1-\delta^{-1}}{\mathbb C} \oplus 
\underset{\delta^{-1}}{L{\mathbb Z}}) * P_2\\
&\cong & LF(2\delta^{-1} -\delta^{-2} + 1 - \frac{1}{n})\\
&\cong & LF(\frac{2}{\sqrt{n}} - \frac{2}{n} +1),
\end{eqnarray*}
where the second isomorphism is a consequence of Proposition \ref{dyk3}.
\end{proof}

\section{Conclusion}

\begin{theorem}\label{main}
Let $H$ be a finite dimensional Kac algebra of dimension $n > 1$, $P = P(H)$ be its planar algebra and $M_0 \subseteq
M_1 \subseteq \cdots$ be the tower of factors associated
to $P$ by the GJS-construction.
Then, $M_0 \cong LF({2n\sqrt{n} - 2n +1})$ and $M_1 \cong LF({2\sqrt{n}-1})$.
\end{theorem}

\begin{proof} The statement about $M_1$ is contained in Theorem \ref{main1}. For $M_0$,
since the tower $M_0 \subseteq
M_1 \subseteq \cdots$ of factors of the GJS-construction is a basic
construction tower with index $n$, the factor $M_2 \cong M_n(M_0)$
or equivalently, $M_0 \cong (M_2)_{\frac{1}{n}}$. By Theorem \ref{main2} and Proposition
\ref{dykrdl}(2) this is computed to be $LF({2n\sqrt{n} - 2n +1})$.
\end{proof}

\begin{remark} If $N \subseteq M$ is a finite index subfactor and
$\alpha > 0$, then the $\alpha$-ampliation subfactor $N_\alpha
\subseteq M_\alpha$ has the same standard invariant (planar
algebra) as $N \subseteq M$. Since all the finite interpolated free group factors $LF(r)$ are ampliations
of each other by Proposition
\ref{dykrdl}(2), our main theorem implies that any $LF(r)$ for
$1 < r < \infty$ is universal
for planar algebras of depth 2, in the sense that given such a planar
algebra it is the planar algebra of a subfactor of $LF(r)$.
\end{remark}

In the light of the previous remark and the results of \cite{PpaShl2003}
on the universality of $LF(\infty)$ it is tempting - and we yield to the
temptation - to conjecture the following.

\begin{conjecture} Any finite interpolated free group factor $LF(r)$
is universal for finite depth subfactor planar algebras.
\end{conjecture}

\section*{Acknowledgements}
We thank Ken Dykema, Krishna Maddaly, Roland Speicher and
Dan Voiculescu for
their prompt and helpful responses to our questions in free probability
theory.

\end{document}